\documentclass[12pt,a4paper]{article}
\usepackage[dvipsnames,svgnames,table]{xcolor}
\usepackage{sectsty}
\usepackage{multirow}
\usepackage{authblk}
\usepackage{relsize}
\usepackage[T1]{fontenc} % Use 8-bit encoding that has 256 glyphs
\usepackage{fourier} % Use the Adobe Utopia font for the document - comment this line to return to the LaTeX default
\usepackage[english]{babel} % English language/hyphenation
\usepackage{amsmath,amsfonts,amsthm, amssymb} % Math packages
\usepackage{graphicx}
\usepackage{booktabs}
\usepackage[font=small]{subcaption}
\usepackage[font=small]{caption}
\usepackage{svg}
\usepackage{wrapfig}
\usepackage{bm}
\usepackage{soul} % Underline
\usepackage{lipsum} % Used for inserting dummy 'Lorem ipsum' text into the template
\usepackage{sectsty} % Allows customizing section commands
\usepackage{fancyhdr} % Custom headers and footers
\usepackage[T1]{fontenc} % Use 8-bit encoding that has 256 glyphs
\usepackage[english]{babel} % English language/hyphenation
\usepackage{framed}
\usepackage{bm}
\usepackage{fancyhdr}
\usepackage{etoolbox}
\usepackage{titlesec}
\usepackage{titletoc}
\usepackage{float}
\usepackage{xcolor}
\theoremstyle{plain}% default
\newcommand{\be}{\begin{equation}}
\newcommand{\ee}{\end{equation}}
\newcommand{\ba}{\begin{equation*}}
\newcommand{\ea}{\end{equation*}}
\newcommand{\beq}{\begin{eqnarray}}
\newcommand{\eeq}{\end{eqnarray}}
\newcommand{\beqs}{\begin{eqnarray*}}
\newcommand{\eeqs}{\end{eqnarray*}}
\newcommand{\rfb}[1]{\mbox{\rm
   (\ref{#1})}\ifx\undefined\stillediting\else:\fbox{$#1$}\fi}
\newtheorem{lemma}{Lemma}[section]
\newtheorem{remark}{Remark}[section]
\newtheorem{definition}{Definition}[section]

\newtheorem{proposition}{Proposition}[section]
\newtheorem{theorem}{Theorem}[section]

\newcommand{\cc}{\Sigma_d}
\newcommand{\ti}{\Sigma_i}
\newcommand{\T}{\mathtt T}

\def\CC{\rm \hbox{C\kern-.56em\raise.4ex
         \hbox{$\scriptscriptstyle |$}\kern+0.5 em }}
         
\title{Energy decay and weak observability for some evolution systems}

\author{
Ka\"is AMMARI\footnote{LR Analyse et Contr\^ole des EDP, LR 22ES03, Universit\'e de Monastir, Facult\'e des Sciences  de Monastir, D\'epartement de Math\'ematiques, 5019, Monastir, Tunisia; Email: kais.ammari@fsm.rnu.tn.} 
\; and\; 
 Faouzi TRIKI\footnote{Faouzi Triki,  Laboratoire Jean Kuntzmann,  UMR CNRS 5224, 
Universit\'e  Grenoble-Alpes, 700 Avenue Centrale,
38401 Saint-Martin-d'H\`eres, France; Email: Faouzi.Triki@univ-grenoble-alpes.fr.}}

\date{\today}

\begin{document}
\maketitle

\begin{abstract}
In an abstract functional framework, we investigate the equivalence between two fundamental properties: weak observability and energy decay associated with certain classes of dissipative operators. This analysis is motivated by the study of evolution equations and their long-time behavior. Our main result establishes that under appropriate assumptions, these two notions are not only closely related but in fact equivalent. The proof relies on recent advances in the theory of weak observability, particularly those presented in \cite{AT1}, where new techniques have been developed to quantify the weak observability inequalities. These tools allow us to bridge the gap between qualitative decay properties and spectral estimates, offering a unified perspective on stability and control in abstract dynamical systems. 
\end{abstract}
 
{\sl {\bf Subjclass[2020]:} 35B35, 35B40, 37L15.}

%\tableofcontents

\section{Introduction} 
An important and longstanding problem in the theory of time-dependent partial differential equations (PDEs) is the analysis of the long-time behavior of their solutions. The main objective is to determine whether these solutions converge to an equilibrium state as time tends to infinity, and if so, to identify the optimal rate of convergence. This question is fundamental in understanding the stability and asymptotic properties of various physical, biological, and engineering systems modeled by evolution equations.

\medskip

In recent years, remarkable progress has been made in this direction through the development of new operator-theoretic approaches. In particular, techniques based on the theory of strongly continuous semi-groups have proven to be highly effective in studying convergence and decay rates. These methods enable a deeper understanding of the interplay between the spectral properties of the governing operators and the dynamical behavior of the solutions. Significant contributions along these lines can be found in \cite{AT, AN, BT, H, Ba, BH, BZ, BD, CPSST}, where various classes of dissipative and non-selfadjoint PDE systems are rigorously analyzed using these advanced tools. 

\medskip

Let $X$ be a complex separable Hilbert space with norm and inner product 
denoted respectively by $\|\cdot\|_{X}$ and  $\langle \cdot,\cdot \rangle_X$.
Let $A : D(A) \subset X\rightarrow X$ be a linear unbounded self-adjoint, 
positive operator with a compact resolvent. Denote by 
${ D}(A^{\frac{1}{2}})$ the domain of 
$A^{\frac{1}{2}}$, and  introduce  for $\beta \in \mathbb R$ the scale of Hilbert
spaces $X_{\beta}$, as follows: for every
$\beta \geq 0$, $X_{\beta}= {D}(A^{\frac{\beta}{2}})$, with the norm
$\|z \|_\beta=\|A^{\frac{\beta}{2}} z\|_X$ (note that 
$0 \notin \sigma(A) $ where $\sigma(A)$
is the spectrum of $A$). The space $X_{-\beta}$ is
defined by duality with respect to the pivot space $X$ as
follows: $X_{-\beta} =X_{\beta}^*$ for $\beta>0$. The operator $A$ can be extended (or restricted) to each $X_\beta$,
such that it becomes a bounded operator
\be
\label{A0extb}
A : X_\beta \rightarrow X_{\beta-2}\;\;\; \forall \beta \in \mathbb R.
\ee

\medskip

Further, let $Y$ be a complex Hilbert space (which will be identified to its dual space) with 
norm and inner product respectively denoted by $||.||_{Y}$ and $\langle \cdot,\cdot \rangle_{Y}$, and let 
$C \in {\mathcal L}(X, Y)$, where  ${\mathcal L}(X, Y)$ the space of linear bounded operators from $X$ into $Y$.

\medskip

We next introduce damped and undamped evolution systems.  
For $(z_0, z_1) \in X_1\times X $, we consider  the following  damped evolution system: 
\be
\label{maineq}
\left\{
\begin{array}{lll}
z^{\prime\prime}(t)+ Az(t) +C^*C z^\prime(t)  = 0, \, t > 0,  \\
z(0) = z_0 \in X_1, \; z^\prime(0) = z_1 \in X,
\end{array}
\right.
\ee
where $C^* \in {\mathcal L}(Y, X)$ is the adjoint of $C \in \mathcal{L}(X,Y)$ (here $Y$ is identified with its dual 
space). The well-posdeness of the Cauchy problem \eqref{maineq}   can be obtained by standard semigroup 
methods \cite{TW09}. 

%%%%%%%%%%%%%%%%%%%%%%%%
\begin{proposition} 
The following results hold.\\
\noindent{\bf (i)} The Cauchy problem \eqref{maineq} has a unique  solution $z \in C([0, +\infty); X_1)\cap C^1([0, +\infty); X)$. \\
\noindent{\bf(ii)} Let $T>0$ be fixed. Then $Cz$ lies in $H^1(0,T; Y)$, and verifies 
\beq
\int_{0}^{T}\|Cz^\prime (t)\|_Y^{2}dt \leq  \kappa_T \|(z_{0}, z_1) \|_{X_1\times X }^2,
\eeq
where the constant $\kappa_T > 0$ is independent of $(z_{0}, z_1) \in X_1\times X$. \\
\noindent{\bf (iii)} $z$ satisfies the energy decay estimate: 
\be
\label{esten}
E(t) - E(0) = -\int_0^t \|Cz^\prime (\tau)\|_Y^2 d\tau,
\ee
where $E(t) = \frac{1}{2} \, \left(\| z(t)\|_{X_1}^2 + \| z^\prime(t)\|_{X}^2 \right).$

\end{proposition}
%%%%%%%%%%%%%
Without loss of generality we assume that $\kappa_T$ is an increasing function of $T$ (if the assumption is not satisfied we substitute  $\kappa_T$ by $\sup_{0\leq t\leq T}\kappa_t$).  
	
\medskip

For $(x_0, x_1) \in X_1\times X $, we consider  the following undamped evolution system: 
\be
\label{maineq2}
\left\{
\begin{array}{lll}
x^{\prime\prime}(t)+ Ax(t) = 0, \, t > 0,  \\
x(0) = x_0 \in X_1, \; x^\prime(0) = x_1 \in X.
\end{array}
\right.
\ee
It is well known that the system \eqref{maineq2} is well-posed \cite{TW09}.  Define 
\beq
\left\{
\begin{array}{ll}
\mathcal A  =\left( \begin{array}{cc} 0 & I \\
-A & 0
\end{array}
\right) : \mathcal{D}(\mathcal{A}):= X_2 \times X_1 \subset X_1 \times X \rightarrow X_1 \times X, \\
\mathcal{C} = ( 0 \; \; C) \in \mathcal{L}(X_1 \times X, Y), \; \mathcal{C}^* = \left( \begin{array}{llcc} 0  \\
C^*
\end{array}
\right) \in \mathcal{L}(Y,X_1 \times X), 
\end{array}
\right.
\eeq
where $I$ is the identity operator.  We next consider the following 
Hilbert space   $X_1\times X$ endowed with the following 
scalar product
\beq 
\left \langle\left( \begin{array}{ll}x_0 \\ x_1 \end{array} \right),
 \left( \begin{array}{ll} z_0 \\ z_1 \end{array}\right) \right \rangle_{X_1\times X}  =  
 \langle A x_0, z_0 \rangle_{X_{-1}, X_1} + \langle x_1, z_1\rangle_{X},
 \eeq
with $A$ here is identified with its unique extension as a bounded operator and invertible
operator from $X_1$  to $X_{-1}$, while  $\langle \cdot, \cdot \rangle_{X_{-1}, X_1} $  is 
the dual product  between $X_{-1}$ and $X_1$. 

\medskip

By denoting $U(t) = \begin{pmatrix} x(t) \\ x^\prime(t) \end{pmatrix}$, 
the system (\ref{maineq2}) can be rewritten as following:

\be
\label{maineq2syst}
\left\{
\begin{array}{lll}
U^{\prime}(t) = \mathcal{A} U(t), \, t > 0,  \\
U(0) = U_0 := \begin{pmatrix} x_0 \\ x_1 \end{pmatrix}.
\end{array}
\right.
\ee

Simple calculation gives, for $(x_0,x_1), (z_0,z_1) \in 
X_2 \times X_1$,
\beq 
\left \langle \mathcal A \left( \begin{array}{ll}x_0 \\ x_1 \end{array} \right),
 \left( \begin{array}{ll}z_0 \\ z_1 \end{array}\right) \right \rangle_{X_1\times X}  =  
 \langle A x_1, z_0 \rangle_{X_{-1}, X_1} -\langle A x_0, z_1 \rangle_{X}\\
 =   \langle x_1,  A z_0 \rangle_{X_{-1}, X_1}  -\langle A x_0,  z_1 \rangle_{X} = - \left
 \langle \left( \begin{array}{ll}x_0 \\ x_1 \end{array} \right), \mathcal A
  \left( \begin{array}{ll}z_0 \\ z_1 \end{array}\right)  \right  \rangle_{X_1\times X},  
\eeq
which implies that $\mathcal A : \mathcal{D}(\mathcal{A}) := X_2 \times X_1 \subset X_1\times X \to X_1\times X$  
is a linear skew-adjoint, with a compact resolvent. We deduce from  Stone's Theorem that 
 $\mathcal A$ generates a strongly continuous unitary group  on 
$X_1\times X$ denoted by  $(e^{t\mathcal A })_{t\in \mathbb R}$.
Then the  system \eqref{maineq2}  has a unique 
solution $x \in  C([0, +\infty); X_1)\cap C^1([0, +\infty); X)$ satisfying 
\beq
 \left( \begin{array}{ll}x(t) \\ x^\prime(t) \end{array} \right) = e^{t \mathcal A }
 \left( \begin{array}{ll}x_0 \\ x_1 \end{array} \right). 
\eeq

Since  $\mathcal A$ is  a skew-adjoint operator with a compact resolvent, it follows that 
the spectrum of $\mathcal A$ is given by  
$\sigma(\mathcal A) \,=\,  \{ i \mu_k, \;  k\in \mathbb Z^* \}$  where $\mu_k \in \mathbb R^* $ is a sequence of 
 increasing  real numbers. In addition, we have $\mu_k= \sqrt{\lambda_k}$ where
 $\lambda_k$ are the eigenvalues of $A$ associated to the eigenfunctions $\varphi_k$. We denote    $(\phi_k)_{k\in
\mathbb Z^*}$  the  orthonormal sequence of eigenvectors of $\mathcal A$ associated to the eigenvalues 
$(i\mu_k)_{k\in
\mathbb Z^*}$.  

%%%%%%%%%%%%%%%%%%
\begin{remark}
Simple calculation shows that $\mu_{-k} = -\mu_k $  for $k \in \mathbb N^*$. In addition 
$\phi_k = \frac{1}{\sqrt 2}\left( \begin{array}{ll} \frac{1}{i\mu_k}\varphi_k \\ \varphi_k \end{array} \right) $
for $k\in \mathbb N^*$ and  $\phi_{-k} =  \frac{1}{\sqrt 2}\left( \begin{array}{ll} \frac{1}{i\mu_k}\varphi_k \\ -\varphi_k \end{array} \right) $ for $k\in \mathbb N^*$, where $(\lambda_k=\mu_k^2, \varphi_k)_{k \in \mathbb N^*}$ are the eigenelements  of 
the operator $A$ (see for instance Proposition 3.7.7 
in \cite{TW09}) satisfying $\langle \varphi_k,  \varphi_l \rangle_{X} = \delta_{kl}$.  
\end{remark}
%%%%%%%%%%%%%%%%%

Using the spectral decomposition of \(A\), we write
\[
\mathcal{A} =\mathcal A_+ - \mathcal A_-,
\]
where
\[
\mathcal A_\pm = \sum_{k=1}^{\infty} \mu_{|k|} \phi_{\pm k}.
\]
We then define the unbounded operator
\[
|\mathcal A| := \mathcal A_+ + \mathcal A_-.
\]
%%%%%%%%%%%%%%%%%%%%%%%%%%%%%%
\begin{definition}
Let $ (x_0, x_1)\in X_2\times X_1 \setminus\{0\}\subset X_1\times X \longmapsto \mu(x_0, x_1)\in \mathbb R^*$ be the $|\mathcal A|$-frequency 
function defined by
\begin{eqnarray}  \label{frequency}
\mu(x_0, x_1) &=& \left \langle - i|\mathcal A| \left( \begin{array}{ll}x_0 \\ x_1 \end{array} \right),
 \left( \begin{array}{ll}x_0 \\ x_1 \end{array}\right) \right \rangle_{X_1\times X}  \left\| \left( \begin{array}{ll}x_0 \\ x_1 \end{array} \right)\right\|_{X_1\times X}^{-2} \\
&=& \sum_{k\in \mathbb Z^*} \mu_{|k|} \left|\left \langle \left( \begin{array}{ll}x_0 \\ x_1 \end{array} \right),
 \phi_k \right \rangle_{X_1\times X} \right|^2
\left( \sum_{k\in \mathbb Z^*} \left|\left \langle \left( \begin{array}{ll}x_0 \\ x_1 \end{array} \right),
 \phi_k \right \rangle_{X_1\times X} \right|^2 \right)^{-1}.
\end{eqnarray} 
\end{definition}
%%%%%%%%%%%%%%%%%%%%%%%%%%%%%%

We observe that $(x_0, x_1) \longmapsto \mu(x_0, x_1)$ is continuous on $X_2\times X_1 \setminus\{0\}$, and $\mu(\phi_k) = \mu_{|k|}, \, k \in \mathbb Z^*$.

\medskip

Let $\cc$ (resp. $\ti$) be the set of functions $\Psi : \mathbb{R} \longrightarrow \mathbb{R}_{+}^{*}$ continuous, even and non-increasing (resp. non-decreasing). Recall that if $\Psi$ $\in$ $\Sigma_d$ is not bounded below by a positive constant it satisfies $\lim_{\lambda \to \pm \infty} \Psi(\lambda) = 0$. Similarly if $\mathtt T  \in \Sigma_i$ is not bounded by a positive constant it verifies $\lim_{\lambda \to \pm \infty} \mathtt T(\lambda) = +\infty$.

\medskip

\begin{definition}
\label{wwbis} 
The system (\ref{maineq2})
is said to be weakly observable  if there
exists $\psi \in \cc $ and $\mathtt T \in \ti$ such that 
  following observation inequality holds:
\begin{eqnarray} 
\label{wobs}
\psi(\mu(x_0, x_1))  \left\| \left( \begin{array}{ll}x_0 \\ x_1 \end{array} \right)\right\|_{X_1\times X}^{2}\leq 
\int_{0}^{\mathtt T(\mu(x_0, x_1))} \left \|\mathcal C e^{t \mathcal A}\left( \begin{array}{ll}x_0 \\ x_1 \end{array} \right) \right\|_{Y}^2 dt
=\int_{0}^{\mathtt T(\mu(x_0, x_1))} \left \| Cx^\prime (t)\right\|_{Y}^2 dt,
 \end{eqnarray}
 for all $ (x_0, x_1)\in X_2\times X_1$. 
If  $\psi(t)$ is lower bounded, and $\mathtt T$ is upper bounded  the system is said  to be exactly observable.
\end{definition}

 \medskip
 
 We further denote 
 
 \[
 w(t) = \left( \begin{array}{ll}z(t) \\ z^\prime(t) \end{array} \right), \textrm{  and   } 
 w(0) = w_0 = \left( \begin{array}{ll}z_0 \\ z_1 \end{array} \right). 
 \]
Consequently the damped evolution equation \eqref{maineq}  can be rewritten as a first order 
system

\be
\label{maineq3}
\left\{
\begin{array}{lll}
w^{\prime}(t)= (\mathcal A- \mathcal{C}^* \mathcal C) w(t), \, t > 0,  \\
w(0) = w_0 \in X_1\times X. 
\end{array}
\right.
\ee
 
Let $\left(T(t)\right)_{t\geq 0}$ be a bounded $C_0$ semi-group on $X_1\times X$, with generator 
$\mathcal A_d := \mathcal A-\mathcal{C}^*\mathcal C$.  Suppose that $i\mathbb R$ is contained in the resolvent set 
$\rho(\mathcal A_d)$ of $\mathcal A_d$.
 Then it is well known that  (see for instance \cite[pp. 40,41]{battybis}, \cite{arentbatty}): 
\begin{eqnarray} \label{stabsemiunif}
\| T(t) \mathcal A_d^{-1}\| \to 0,  \textrm{  as  } t \to +\infty.
\end{eqnarray}

For the sake of completeness, we provide in the appendix, see Theorem \ref{thmstab}, an elegant proof due to Batty \cite{battybis}, based on an argument originally developed by Korevaar \cite{korevaar}.

\medskip

The aim of this paper is to characterize the non-uniform decay properties of the damped problem \eqref{maineq} via weak observability inequalities for the conservative problem \eqref{maineq2}. More precisely, our objective  is to study  the relation between 
the  rate of  decay of $\| T(t) \mathcal A_d^{-1}\| $ 
and the weak observability \eqref{wobs}  in terms of   the  operator $C$ and the 
spectral properties of the operator $A$. This problem is of importance in many  applications 
 related to the evolution of dynamic systems. The decay rate of the semigroup provide some hints on how fast 
  these systems will reach their equilibrium. Many works already related  the behavior of $\| T(t)  (\mathcal A-\mathcal{C}^*\mathcal C)^{-1}\| $  for large
   time  to the growth of the resolvent $R(i\mu, \mathcal A_d) = (i\mu I -\mathcal A_d)^{-1}$ on the imaginary axis \cite{ AT, Ba, BT, BH, BZ,lyubich, BD}. \\
 
 Let $\Sigma_d^0$ be the set of functions $\mathcal{K} : \mathbb{R_+} \longrightarrow \mathbb{R}_{+}^{*}$ continuous, decreasing and satisfying 
 $\lim_{t \to +\infty} \mathcal K (t) = 0$.
\begin{proposition}
\label{ww} 
There exists $\mathcal K \in \Sigma_d^0$ such that the following inequality holds:
\begin{eqnarray} 
\label{stability}
\| T(t)  w_0\|_{X_1\times X}^2  \leq  \mathcal K(t) g(\mu(w_0)) \|w_0\|_{X_1\times X}^2,\qquad \forall w_0 \in X_2\times X_1\setminus \left\{0\right\},
\end{eqnarray}
 where $g(s)=1$ for uniform stability, and $g(s)=s$ for non-uniform stability.
\end{proposition}
 The proof of this result is provided in the Appendix. 
   
%%%%%%%%%%%%%%%%%%
\section{Main results} \label{sec2}
We present in this section the  main results of our paper.

\begin{theorem} \label{wstab}
 We assume that the weak observability inequality (\ref{wobs}) is satisfied for some $\psi \in \cc $ and $\mathtt T \in \ti$. 
 Let $w$ be the solution to the damped 
 system with initial condition $w_0 \in X_2 \times X_1 \setminus \left\{0\right\}$, and define the normalized energy:
 \beqs
 \mathcal{H}(t) := \frac{\left\|w(t)\right\|^2_{X_1 \times X}}{\left\|w_0\right\|^2_{\mathcal{D}(\mathcal{A}_d)}},\quad \forall t>0.
 \eeqs
 Then, there exists $C >0$ such that $\mathcal H$ satisfies the  following inequality: 
\be
\mathcal{H}(t)  \leq \frac{C}{(\mathcal{G}(\mathcal{H}(t)))^2} 
\left( \mathcal{H}(t)  - \mathcal{H}\left(t + \frac{1}{\mathcal{G}(\mathcal{H}(t))} \right)
 \right), \qquad \forall \, t > 0,
\label{nunif}
\ee
where $\mathcal G: \mathbb R_+^* \to \mathbb R_+$ is the function in $\ti$ defined by $ \mathcal{G}(r) := \frac{\sqrt{\psi(1/r)}}{\T(1/r)}$,  and $\left\|w_0\right\|^2_{\mathcal{D}(\mathcal{A}_d)} := \left\|w_0\right\|^2_{X_1 \times X} + \left\|\mathcal{A}_d w_0\right\|^2_{X_1 \times X}$.
\end{theorem}
\begin{proof}
We begin the proof of the theorem by deriving a weak observability inequality for the damped system \eqref{maineq3}, which is of independent interest.

\begin{lemma}
The solution $w$  of  the damped system \eqref{maineq3} with initial condition $w_0 \in X_2 \times X_1 \setminus \left\{0\right\}$, satisfies 
\begin{eqnarray} \label{wobsdamped}
\widetilde \psi(\mu(w_0)) \left\|w_0\right\|^2_{X_1 \times X} \leq \int_0^{\T(\mu(w_0))}  \left\|\mathcal{C}w(s)\right\|^2_{Y} \, ds,
\end{eqnarray}
where $\widetilde \psi \in \cc$ is defined by
\begin{equation*}
\widetilde \psi(\mu)= \frac{1}{2}\left(1 + 4 \T^2 (\mu ) 
\left\|\mathcal{C}\right\|^4_{\mathcal{L}(X_1 \times X,Y)}
\right)^{-1} \psi(\mu).
\end{equation*}

\end{lemma}
\begin{proof}
Let $v(t) = U(t) - w(t)$, where $U(t)$ is solution of \rfb{maineq2syst} with $U(0) = w_0$, and $w(t)$ satisfies \rfb{maineq3}. So, $v(t)$ verifies
\be
\label{maindiff}
\left\{
\begin{array}{ll}
v^\prime(t) = \mathcal{A} v(t) + \mathcal{C}^* \mathcal{C} w(t),  \, t > 0, \\
v(0) = 0.
\end{array}
\right.
\ee
Multiplying \rfb{maindiff} by $v(t)$, we get 
$$
\left\langle v^\prime(t),v(t)\right\rangle_{X_1 \times X} = \left\langle \mathcal{A} v(t),v(t)\right\rangle_{X_1 \times X} + 
\left\langle \mathcal{C} w(t), \mathcal{C} v(t)\right\rangle_Y.
$$
Which implies 
$$
\Re \left\langle v^\prime (t),v(t)\right\rangle_{X_1 \times X} =  \Re \left\langle \mathcal{C} w(t),\mathcal{C} v(t)\right\rangle_Y.
$$
Then
\begin{equation} \label{eq:keyidentity}
\frac{1}{2} \, \left\| v(t)\right\|^2_{X_1 \times X} =  \int_0^t \Re \left\langle \mathcal{C} w(s),\mathcal{C} v(s)\right\rangle_Y \, ds, \,  \forall t > 0,
\end{equation}
Hence 
$$
\left\| v(t)\right\|^2_{X_1 \times X} \leq \int_0^t \left(\frac{1}{\delta} \, \left\|\mathcal{C} w(s)\right\|^2_Y + \delta \, \left\| \mathcal{C} v(s)\right\|^2_Y \right) \, ds, \, t > 0,\, \forall \delta>0.
$$
Let $$\mathcal{M}(t) := \sup_{0 < s <t} \left\|v(s)\right\|^2_{X_1 \times X}.$$ Then
$$
\mathcal{M}(t) \leq \frac{1}{\delta} \, \int_0^t \left\|\mathcal{C} w(s)\right\|^2_Y \, ds + \delta T \left\|\mathcal{C}\right\|^2_{\mathcal{L}(X_1 \times X,Y)} \mathcal{M}(T), \, 0 < t < T,\, \delta > 0. 
$$
So, for $\delta = \frac{1}{2T \left\|\mathcal{C}\right\|^2_{\mathcal{L}(X_1 \times X,Y)}}$, we get 
$$
\mathcal{M}(T) \leq 4 T \left\|\mathcal{C}\right\|^2_{\mathcal{L}(X_1 \times X,Y)} \int_0^T \left\|\mathcal{C}w(s)\right\|^2_{Y} \, ds.
$$
Therefore
$$
\int_0^T \left\|v(s)\right\|^2 \, ds \leq T 
\mathcal{M}(T) \leq 4 T^2  \left\|\mathcal{C}\right\|^2_{\mathcal{L}(X_1 \times X,Y)} \, 
\int_0^T \left\|\mathcal{C}w(s)\right\|^2_{Y} \, ds,
$$
which gives
$$
\int_0^T \left\|\mathcal{C}U(s)\right\|^2_{Y} \, ds \leq 2 \int_0^T \left\|\mathcal{C}v(s)\right\|^2_{Y} \, ds + 
2 \int_0^T \left\|\mathcal{C}w(s)\right\|^2_{Y} \, ds 
$$
$$
\leq 2 \left( 1 + 4 T^2  \left\|\mathcal{C}\right\|^4_{\mathcal{L}(X_1 \times X,Y)}\right) \int_0^T \left\|\mathcal{C}w(s)\right\|^2_{Y} \, ds.
$$
Therefore, by the weak observability inequality \rfb{maineq2syst} satisfied by $U(t)$, we get 
$$
\psi(\mu(w_0)) \left\|w_0\right\|^2_{X_1 \times X} \leq 2 \left( 1 + 4 \T^2 (\mu(w_0))  \left\|\mathcal{C}\right\|^4_{\mathcal{L}(X_1 \times X,Y)}\ \right) \, \int_0^{\T(\mu(w_0))}  \left\|\mathcal{C}w(s)\right\|^2_{Y} \, ds,
$$
which finishes the proof of the lemma.

\end{proof}

Back to the proof of Theorem \ref{wstab}. Combining the estimates \eqref{wobsdamped} and \eqref{esten}, we get
$$
\psi(\mu(w_0)) \left\|w_0\right\|^2_{X_1 \times X} \leq 2 \left( 1 + 4 \T^2 (\mu(w_0))  \left\|\mathcal{C}\right\|^4_{\mathcal{L}(X_1 \times X,Y)}\ \right) \, \left(\left\|w_0\right\|^2_{X_1 \times X} - \left\|w(\T(\mu(w_0))\right\|^2_{X_1 \times X} \right),
$$
for all $w_0 \in X_2 \times X_1 \setminus \left\{0\right\}$. \\

Applying the previous inequality to $w(t)$ on the time interval $\left(t, t+\T(\mu(w(t)))\right)$, we obtain
$$
\psi(\mu(w(t))) \left\|w(t)\right\|^2_{X_1 \times X} \leq 
$$
\be
\label{**}
2 \left( 1 + 4 \T^2 (\mu(w(t)))  \left\|\mathcal{C}\right\|^4_{\mathcal{L}(X_1 \times X,Y)}\ \right) \, \left(\left\|w(t)\right\|^2_{X_1 \times X} - \left\|w(t+\T(\mu(w(t)))\right\|^2_{X_1 \times X} \right), \, \forall \, t > 0.
\ee
Since $\mu(w(t)) \left\|w(t)\right\|^2_{X_1 \times X} = \left\langle -i \mathcal{A} w(t),w(t)\right\rangle_{X_1 \times X}  \in \mathbb{R}$,
multiplying \rfb{maineq3} by $w(t)$, leads to
$$
\left\langle w^\prime(t),w(t) \right\rangle_{X_1 \times X} = \left\langle \mathcal{A} w(t),w(t)\right\rangle_{X_1 \times X} - \left\|\mathcal{C}w(t)\right\|^2_{Y}.
$$
Then 
$$
\left\langle -i \mathcal{A} w(t),w(t)\right\rangle_{X_1 \times X} = \Im \left\langle w^\prime(t),w(t) \right\rangle_{X_1 \times X}
$$
which gives
$$
\mu(w(t)) \left\|w(t)\right\|^2_{X_1 \times X} \leq \frac{1}{2} \left\|w^\prime(t)\right\|^2_{X_1 \times X} + \frac{1}{2} \left\|w(t)\right\|^2_{X_1 \times X}.
$$
We have that $Z(t) = w^\prime(t), w_0 \in X_2 \times X_1,$ satisfies 
$$
\left\{
\begin{array}{ll}
Z^\prime(t) = \mathcal{A}_d Z(t), \, t > 0, \\
Z(0) = w^\prime (0) = \mathcal{A}_d w_0.
\end{array}
\right.
$$
Then, 
$$\left\|Z(t)\right\|^2_{X_1 \times X} - \left\|Z(0)\right\|^2_{X_1 \times X} = -2 \int_0^t \left\|\mathcal{C}Z(s)\right\|^2_Y \, ds.$$ 
Hence, 
\be
\label{prop}
\left\|w^\prime(t)\right\|^2_{X_1 \times X} \leq  \left\|Z(0)\right\|^2_{X_1 \times X} = \left\|\mathcal{A}_d w_0\right\|^2_{X_1 \times X}.
\ee
Then \rfb{**} and \rfb{prop} imply that
$$
\left\|w(t)\right\|^2_{X_1 \times X} \leq  \frac{1 + 4 \left\|\mathcal{C}\right\|^4_{\mathcal{L}(X_1 \times X,Y)} 
\T^2 \left( 
\frac{\left\|w_0\right\|^2_{\mathcal{D}(\mathcal{A}_d)}}{\left\|w(t)\right\|^2_{X_1 \times X}} 
\right)}{\psi \left( \frac{\left\|w_0\right\|^2_{\mathcal{D}(\mathcal{A}_d)}}{\left\|w(t)\right\|^2_{X_1 \times X}} \right)} \times 
$$
$$
\left( \left\|w(t)\right\|^2_{X_1 \times X} - 
\left\|w\left(t + \T \left(\frac{\left\|w_0\right\|^2_{\mathcal{D}(\mathcal{A}_d)}}{\left\|w(t)\right\|^2_{X_1 \times X}}\right) 
\right)\right\|^2_{X_1 \times X}\right), \, \forall \, t > 0.
$$
We note that  $\mathcal{H}$ (resp. $\mathcal{G}$) is a continuous positive decreasing (resp. increasing)
real function on $]0,+\infty)$ and $\mathcal{H}$ is bounded by one.

Then, we have
$$
\mathcal{H}(t) \leq  \frac{1 + 4 \left\|\mathcal{C}\right\|^4_{\mathcal{L}(X_1 \times X,Y)} 
\T^2 \left(1/\mathcal{H}(t)\right)}{\psi \left(1/\mathcal{H}(t)\right)} \times 
\left(\mathcal{H}(t) - \mathcal{H}\left(t + \T \left(1/\mathcal{H}(t)\right)\right)\right), \, \forall \, t > 0,
$$

Therefore 

$$
\mathcal{H}(t) \leq  C \, 
\frac{1}{\left(\mathcal{G}(\mathcal{H}(t)\right))^2} \times 
\left(\mathcal{H}(t) - \mathcal{H}\left(t + \frac{1}{\mathcal{G}(\mathcal{H}(t))}\right)\right), \, \forall \, t > 0,
$$
where $C = 1/\T(1) + 4 \left\|\mathcal{C}\right\|^4_{\mathcal{L}(X_1 \times X,Y)}$, which implies the desired inequality  \rfb{nunif}.
\end{proof}
The following results constitute an improvement of the results derived in \cite[Lemma 2.1]{ABE}. 
\begin{lemma} \label{Mo2}
Let $\mathcal H: \mathbb R_+\to \mathbb R_+^*$ (resp. $\mathcal G: [0,1]\to \mathbb R_+$) be a continuous  decreasing (resp. non-decreasing) real function on $[0, +\infty)$. Let 
$\kappa$ be a positive constant verifying $ \kappa > \mathcal G^2(1)$.  Suppose that $\mathcal H$ is bounded by one and satisfies 
\be 
\mathcal{H}(t)  \leq \kappa(\mathcal{G}(\mathcal{H}(t)))^{-2} 
\left( \mathcal{H}(t)  - \mathcal{H}\left(t + \left(\mathcal{G}(\mathcal{H}(t))\right)^{-1} \right)\right), \, \forall \, t > 0.
\label{nunifbis}
\ee
\begin{itemize}
\item[(i)]  If \, $\mathcal G(r) \geq  \mathcal G_0 $ for all $r \in (0, 1)$   $\mathcal G_0$ is some positive 
constant, then
\beq \label{Gr1}
\mathcal H(t) \leq e^{\kappa^{-1}\mathcal G_0^2} e^{-\kappa^{-1}\mathcal G_0^3  t }\mathcal H(0), \quad \forall t>0.
\eeq
\item[(ii)]   If \, $\lim_{t\to 0^+ }\mathcal G(t) =0 $, then  

\beq \label{Gr12}
\mathcal H(t) \leq  \mathcal F^{-1} \left(\frac{2\kappa}{t}\right)\mathcal H(0), \quad \forall t>0,
\eeq
where $\mathcal F: \mathbb R_+ \rightarrow \mathbb R_+$ is the function defined by $\mathcal F(r) = r\mathcal G^3(r),$ and $ \mathcal F^{-1}$ is its pseudo-inverse function 
defined by $\mathcal F^{-1}(r) = \inf\{s \in \mathbb R_+; \mathcal F(s)>r\}$. 
\end{itemize}
\end{lemma}
%%%%%%%%%%%%%%%%%%%%%%%%%%%%%%%%%%%%%%%%%%%%%%%%%%%%%%%%%%%%%%%%%%%%%%%%

The proof of the Lemma above is provided in the Appendix.
As a consequence of Lemma \ref{Mo2}, we have the following stability estimate.

\begin{theorem} \label{mainstab} 
We assume that the weak observability inequality (\ref{wobs}) is satisfied for some $\psi \in \cc $ and $\mathtt T \in \ti$. 
Let  $\mathcal G: \mathbb R_+ \to \mathbb R_+$ be the function  $ \mathcal{G}(x) := \frac{\sqrt{\psi(1/x)}}{\T(1/x)}$. 

\begin{itemize}
\item[(i)] If \, $\mathcal G(r) \geq  \mathcal G_0 $ for all $r \in (0, 1)$  and $\mathcal G_0$ is some positive 
constant, then
\beq \label{Gr13}
\left\|w(t)\right\|^2_{X_1 \times X}  \leq e^{\kappa^{-1}\mathcal G_0^2} e^{-\kappa^{-1}\mathcal G_0^3  t }\left\|w_0\right\|^2_{X_1\times X}, \quad \forall \, t > 0, \,  \forall w_0 \in X_2\times X_1.
\eeq

\item[(ii)]   If \,  $\lim_{t\to 0^+ }\mathcal G(t) =0 $, then  

\beq \label{Gr14}
\left\|w(t)\right\|^2_{X_1 \times X} \leq  \mathcal F^{-1} \left(\frac{2\kappa}{t}\right) \mu(w_0)\left\|w_0\right\|^2_{X_1\times X},  \quad \forall \, t > 0, \, \forall w_0 \in X_2\times X_1,
\eeq
where $\mathcal F: \mathbb R_+ \rightarrow \mathbb R_+$ is the function defined by $\mathcal F(r) = r\mathcal G^3(r),$ and $ \mathcal F^{-1}$ is its pseudo-inverse function 
defined by $\mathcal F^{-1}(r) = \inf\{s \in \mathbb R_+; \mathcal F(s)>r\}$. 
 \end{itemize}
\end{theorem}
%%%%%%%%%%%%%%%%%%%
\begin{proof}
The results of the theorem are direct consequence of  Theorem \ref{wstab} and  Lemma \ref{Mo2}. 
\begin{itemize}
\item[(i)] If \, $\mathcal G(r) \geq  \mathcal G_0 $, considering the function $\mathcal H(t)$
of Theorem \ref{wstab}, and the first case in Lemma \ref{Mo2}, we obtain 
\[
\mathcal H(t) \leq e^{\kappa^{-1}\mathcal G_0^2} e^{-\kappa^{-1}\mathcal G_0^3  t }\mathcal H(0), \quad \forall t>0.
\]
Multiplying both terms in the inequality by $\left\|w_0 \right\|^2_{\mathcal{D}(\mathcal{A}_d)}$,  we get the desired result.

\item[(ii)]   If \,  $\lim_{t\to 0^+ }\mathcal G(t) =0 $,  again taking the function $\mathcal H(t)$
of Theorem \ref{wstab}, and considering the second case in Lemma \ref{Mo2}, we obtain 
\[ 
\mathcal H(t) \leq  \mathcal F^{-1} \left(\frac{2\kappa}{t}\right) \mathcal H(0),  \quad \forall \, t > 0.
\]
Since $\mathcal H(0) \leq 1$, we get 
\[
\left\|w(t)\right\|^2_{X_1 \times X} \leq  \mathcal F^{-1} \left(\frac{2\kappa}{t}\right)\left\|w_0\right\|^2_{X_2\times X_1},  \quad \forall \, t > 0, \, \forall w_0 \in X_2\times X_1,
\]
which achieves the proof of the theorem.
\end{itemize}
\end{proof}

%%%%%%%%%%%%%%%%%%%%%%%%
\begin{remark}
It is well known that the system \eqref{maineq} is exactly observable if and only if $\mathcal G(r) \geq \mathcal G_0$ for sufficiently small $r$. Therefore, the exponential energy decay obtained for the system \eqref{maineq3} is consistent with the well-established results in the literature \cite{AT}.
\end{remark}

Next we shall investigate the reciprocity of  Theorem \ref{mainstab}, i.e. uniform/non-uniform stability  implies  exact/weak observability. Let 
$\mathcal K \in \Sigma_d^0$, and assume that we have the following stability:
\beq \label{stab}
\left\|w(t)\right\|^2_{X_1 \times X} \leq  \mathcal K(t) g(\mu(w_0)) \left\|w_0\right\|^2_{X_1\times X},  \quad \forall \, t > 0, \, w_0 \in X_2\times X_1,
\eeq
with as in Proposition \ref{ww}, $g(\mu) = 1$ for uniform stability and $g(\mu) = \mu$ in case of non-uniform stability.

\begin{theorem} \label{stabobs}
We suppose that \rfb{stab} holds.

\begin{itemize}
    \item In the case of uniform stability ($g(\mu)=1$), the following exact observability inequality
    \begin{equation}
        \|w_0\|_{X_1\times X}^2
        \leq
        4 \int_0^t \|\mathcal{C}U(s)\|_Y^2\, ds,
        \qquad
        \forall\, w_0 \in X_2 \times X_1,
    \end{equation}
    is satisfied for all $t \geq T_0$, where
    \[
        T_0 = \mathcal{K}^{-1}\!\left(\frac12\right).
    \]

    \item In the case of non-uniform stability ($g(\mu)=\mu$), the following weak observability inequality
    \begin{equation} \label{obsnunif}
        \|w_0\|_{X_1\times X}^2
        \leq
        4 \int_0^t \|\mathcal{C}U(s)\|_Y^2\, ds,
    \end{equation}
    is verified for all
    \[
        t \geq \mathcal{T}(\mu(w_0)),
    \]
    where $\mathcal{T}\in\Sigma_i$ is defined by
    \[
        \mathcal{T}(\mu(w_0))
        =
        \mathcal{K}^{-1}\!\left(\frac{1}{2\mu(w_0)}\right).
    \]

\end{itemize}
Here $ \mathcal K^{-1}$ is the pseudo-inverse function 
defined by $\mathcal K ^{-1}(r) = \inf\left\{s \in \mathbb R_+; \mathcal K(s)>r\right\}$. 

\end{theorem}

%%%%%%%%%%%%%%%%%%%%
\begin{proof}
Let $v(t) = U(t) - w(t)$, where $U(t)$ is solution of \rfb{maineq2syst} with $U(0) = w_0$, and $w(t)$ satisfies \rfb{maineq3}.
We first remark that inequality \eqref{eq:keyidentity} leads to
$$
\left\| v(t)\right\|^2_{X_1 \times X} + \int_0^t \|\mathcal{C} v(s)\|_Y^2\, ds  \leq  \int_0^t \|\mathcal{C} U(s)\|_Y^2 ds, \quad \forall t > 0.
$$
Therefore 
\begin{equation} \label{eq:compbis1}
\int_0^t \|\mathcal{C} v(s)\|_Y^2\, ds  \leq  \int_0^t \|\mathcal{C} U(s)\|_Y^2 ds, \quad \forall t > 0.
\end{equation}
Consequently 
\begin{equation} \label{eq:comp}
\int_0^t \|\mathcal{C} w(s)\|_Y^2\, ds  \leq  \int_0^t \|\mathcal{C} U(s)\|_Y^2 ds+ \int_0^t \|\mathcal{C} v(s)\|_Y^2 ds \leq 2\int_0^t \|\mathcal{C} U(s)\|_Y^2 ds,  \quad \forall t > 0.
\end{equation}
On the other hand, the equality \eqref{esten} is equivalent to 
\beq \label{eq:energy00}
\int_0^t \|\mathcal{C} w(s)\|_Y^2 = \left\|w_0\right\|^2_{X_1\times X}- \left\|w(t)\right\|^2_{X_1 \times X}, \quad  \forall t > 0.
\eeq
Combining \eqref{eq:comp} and \eqref{eq:energy00} gives 

\begin{equation} \label{eq:compbis2}
 \left\|w_0\right\|^2_{X_1\times X}- \left\|w(t)\right\|^2_{X_1 \times X} \leq  \int_0^t \|\mathcal{C} U(s)\|_Y^2 ds+ \int_0^t \|\mathcal{C} v(s)\|_Y^2 ds \leq 2\int_0^t \|\mathcal{C} U(s)\|_Y^2 ds,  \quad \forall t > 0.
\end{equation}

We  suppose that \rfb{stab} holds, then we have two cases: 

\begin{itemize}
\item 
If \, $ g(\mu)= 1$, we have 

$$
(1-\mathcal K(t))\left\|w_0\right\|^2_{X_1\times X} \leq 2\int_0^t \|\mathcal{C} U(s)\|_Y^2 ds,  \quad \forall t > 0.
$$

Then, the following observability estimate

$$
\left\|w_0\right\|^2_{X_1\times X} \leq 4 \int_0^t \|\mathcal{C} U(s)\|_Y^2 ds, \quad \forall w_0 \in X_2 \times X_1,
$$

holds for all $t\geq T_0$, with $T_0 = \mathcal K^{-1}\left(\frac{1}{2}\right)$.

\item 
If \, $ g(\mu)= \mu$, we have 

$$
(1-\mathcal K(t) \mu(w_0))\left\|w_0\right\|^2_{X_1\times X} \leq 2\int_0^t \|\mathcal{C} U(s)\|_Y^2 ds,  \quad \forall t > 0.
$$
Hence 
$$
\left\|w_0\right\|^2_{X_1\times X} \leq 4\int_0^t \|\mathcal{C} U(s)\|_Y^2 ds,  
$$
for $t\geq \T(\mu(w_0))$ where $\T(\mu(w_0)) = \mathcal K^{-1}\left(\frac{1}{2\mu(w_0)}\right)$. Since $\mathcal K$ is a decreasing, function 
$\T \in \Sigma_i$. 
\end{itemize}
\end{proof}

%%%%%%%%%%%%%%%%%%%%%%%%%%%%%%%%%%%%%%%%%%%%%%%%%%%%%%%%%%%%%%%%
\section{Application to the wave equation}

Let $\Omega = (0,1)^2$ be the unit square and $a$ the
damping coefficient, which depends only on $x \in (0,1)$ such that
$a(x) = a_0 > 0$ for $x < \sigma$ and
$a(x) = 0$ for $x > \sigma$, where $\sigma$ is some fixed number of the interval $(0,1)$. We consider the damped wave
equation:
\be
\label{wd}
\left\{
\begin{array}{ll}
u_{tt}(t,x,y) - \Delta u(t,x,y) + a(x) \, u_t(t,x,y) = 0, t \in (0,+\infty), (x,y) \in \Omega,\\
u(t, x, y) = 0,  t \in (0,+\infty), (x, y) \in \partial \Omega, \\
u(0, x, y) = u_0(x, y), u_t(0, x, y) = u_1(x, y),  (x, y) \in \Omega.
\end{array}
\right.
\ee
Here, 
$$
X = L^2(\Omega), X_1 = H^1_0(\Omega), X_2 = H^2(\Omega) \cap H^1_0(\Omega), Y = L^2(\Omega),
$$
and
$$
A = - \Delta, \,   C =  C^* = \sqrt{a} \, I \in \mathcal{L}(L^2(\Omega)).
$$
According to \cite{S}, the system \eqref{wd} is not exponentially stable but enjoys a polynomial decay rate. More precisely, there exists a constant $c_0 >0$ such that
\begin{equation} \label{estimatestabb}
\left\|(u(t),u_t(t))\right\|_{X_1\times X}^{2}
\leq
\frac{c_0}{t^{4/3}}
\left\|(u_0,u_1)\right\|_{X_2\times X_1}^{2},
\qquad \forall\, t>0.
\end{equation}

Consequently  \eqref{wd} satisfies the non-uniform stability estimate  \eqref{stab}  with 
\begin{equation*} 
g(\mu):= \mu, \quad \textrm{and} \quad\mathcal{K}(t) = \frac{c_0}{t^{4/3}}, \;\; \;\text{for} \;\;t>0.
\end{equation*}
Therefore, by Theorem~\ref{stabobs}, the above polynomial stability estimate implies a  weak observability inequality of the form \eqref{obsnunif}. More precisely,
for $(v_0,v_1) \in X_2\times X_1$, let 
$v$ be the solution to the following undamped wave equation:
\be
\label{wdbis}
\left\{
\begin{array}{ll}
v_{tt}(t,x,y) - \Delta v(t,x,y)  = 0, t \in (0,+\infty), (x,y) \in \Omega,\\
v(t, x, y) = 0,  t \in (0,+\infty), (x, y) \in \partial \Omega, \\
v(0, x, y) = v_0(x, y), v_t(0, x, y) = v_1(x, y),  (x, y) \in \Omega.
\end{array}
\right.
\ee
Then $v$ satisfied the following weak observability estimate: 

\[
\Psi(\mu(v_0,v_1))\left\|(v_0,v_1)\right\|_{X_1\times X}^{2} \leq  \int_0^T \| Cv_t(t)\|_{Y}^2 dt,\qquad \forall T\geq {\tt T}(\mu(v_0,v_1)),
\]
where the functions $\Psi \in \Sigma_d$ and ${\tt T} \in \Sigma_i$ are given by
\[
\Psi(\mu) = \frac{1}{4}, \qquad {\tt T}(\mu) = (2c_0\mu)^{\frac{3}{4}}.
\]

%Hence, the decay rate $t^{-4/3}$ of the energy is directly related to a weak observability estimate through the stability-observability correspondence established in Theorem~\ref{stabobs}.

The approach developed in this section readily extends to a wide range of settings in
which stability  estimates of the form \eqref{estimatestabb} are available.
 
%%%%%%%%%%%%%%%%%%%%%%%%

%%%%%%%%%%%%%%%%%%%%%%%%%%%%%%%%%%%%%%%%%%
\section{Appendix} \label{appendix}

%%%%%%%%%%%%%%%%%%%%%%%%%%%%%%%%%%%%%%%%%%%%%%%%%%%%%%%%%

\begin{proof}[Proof of Lemma \ref{Mo2}]
Dividing the inequality \eqref{nunifbis} both sides by $\mathcal H$, we get
\beq \label{Gre3}
\left(\mathcal H(t)\right)^{-1}\mathcal{H}\left(t + \left(\mathcal{G}(\mathcal{H}(t))\right)^{-1}\right) \leq 1-\kappa^{-1}\mathcal{G}^2(\mathcal{H}(t)).
\eeq
Define the real sequence $(t_n)_{n\in \mathbb N}$   by 
\beq \label{Gre01}
\left\{ \begin{array}{cc}
    t_{n+1} = t_n +  & \left(\mathcal{G}(\mathcal{H}(t_n))\right)^{-1},\;  n\in \mathbb N,\\
     t_0= 0.& 
\end{array}
\right.
\eeq 
Since $\mathcal H(t)>0$ for all $t\in \mathbb R_+$, the sequence  $(t_n)_{n\in \mathbb N^*}$ is increasing, positive and satisfies
$\lim_{n\to +\infty} t_n = +\infty$. In addition, we have 
$$
t_n = \sum_{j=0}^{n-1}\left(\mathcal{G}(\mathcal{H}(t_j))\right)^{-1}.
$$
We deduce from the monotony of $\mathcal G$ and $\mathcal H$ the following estimates
\beq \label{Gre4}
 n \left(\mathcal{G}(\mathcal{H}(0))\right)^{-1}\leq t_n \leq n \left(\mathcal{G}(\mathcal{H}(t_n))\right)^{-1}, \; \forall n\in \mathbb N. 
\eeq 
Back to the inequality \eqref{Gre3},we have 
\beq \label{Gre5}
\mathcal H(t_{n+1}) \leq \left( 1-\kappa^{-1}\mathcal{G}^2(\mathcal{H}(_n))\right)\mathcal H(t_n), \quad \forall n\in \mathbb N.
\eeq
Hence 
\beqs
\mathcal H(t_{n}) \leq \prod_{j=0}^{n-1} \left( 1-\kappa^{-1}\mathcal{G}^2(\mathcal{H}(t_j))\right) \mathcal H(0),  \quad \forall n\in \mathbb N^*.
\eeqs
Therefore 
\beqs
\mathcal H(t_{n}) \leq \left( 1-\kappa^{-1}\mathcal{G}^2(\mathcal{H}(t_n))\right)^{n} \mathcal H(0),  \quad \forall n\in \mathbb N^*.
\eeqs
Using now inequality \eqref{Gre4}, leads to
\beq \label{Gre6}
\mathcal H(t_{n}) \left( 1-\kappa^{-1}\mathcal{G}^2(\mathcal{H}(t_n))\right)^{-t_n \mathcal{G}(\mathcal{H}(t_n))}\mathcal H(0) \leq 1,  \quad \forall n\in \mathbb N^*.
\eeq
Forward calculation gives 
\beqs
e^{\kappa^{-1} t_n x^3}\leq (1-\kappa^{-1} x^2)^{-t_n x}, \quad \forall x\in (0, \sqrt \kappa).  
\eeqs
Combining the previous inequality with \eqref{Gre6}, we find 
\beq \label{Gre7}
\mathcal H(t_{n})e^{\kappa^{-1} t_n \mathcal{G}^3(\mathcal{H}(t_n))}  \leq \mathcal H(0),  \quad \forall n\in \mathbb N^*.
\eeq
\begin{itemize}
\item[(i)] Since $\mathcal G(H(t_n)) \geq  \mathcal G_0 $, we obtain from \eqref{Gre7} the following estimate
\beq \label{Gre8bis}
\mathcal H(t_{n}) \leq e^{- \kappa^{-1}\mathcal{G}_0^3 t_n} \mathcal H(0),  \quad \forall n\in \mathbb N^*.
\eeq
On the other hand we deduce from \eqref{Gre01} that $$t_{n+1} - t_n = \left(\mathcal{G}(\mathcal{H}(t_n))\right)^{-1} 
\leq \mathcal G_0^{-1}.$$  

Therefore for a given $t\in (0, +\infty)$ there exists a unique $n^*\in \mathbb N$ such that
$t_{n^*} \in (t-\mathcal G_0^{-1}, t]$.\\

Consequently 
\beqs
\mathcal H(t) \leq \mathcal H(t_{n^*}) \leq e^{- \kappa^{-1}\mathcal{G}_0^3 t_{n^*}}\mathcal H(0) 
\leq e^{- \kappa^{-1}\mathcal{G}_0^3(t-\mathcal G_0^{-1})}\mathcal H(0),
\eeqs
which gives the desired result.

\item[(ii)] Now we assume that  $\lim_{t\to 0^+ }\mathcal G(t) =0 $. Using the fact that
$e^{x} \geq 1+x$ for all $x \in \mathbb R_+$, we deduce from \eqref{Gre7} the following estimate 

\beq \label{Gre8}
\mathcal H(t_{n}) \kappa^{-1}t_n \mathcal{G}^3(\mathcal{H}(t_n))  \leq \mathcal H(0),  \quad \forall n\in \mathbb N^*. 
\eeq

Let $t>0$ be fixed. Since the sequence $(t_n)_{n\in \mathbb N}$ is increasing there exists a unique $n^*\in \mathbb N$ verifying $t\in [t_{n^*}, t_{n^*+1})$

Considering the monotony of the functions $\mathcal G$ and $\mathcal H$, we get
\beqs
t_{n^*+1} - t_{n^*} = \left(\mathcal{G}(\mathcal{H}(t_{n^*}))\right)^{-1} \leq \left(\mathcal{G}(\mathcal{H}(t))\right)^{-1}. 
\eeqs
Therefore  $t-  \left(\mathcal{G}(\mathcal{H}(t))\right)^{-1}>0$, and we have 
\beq  \label{Gre9}
t_{n^*} \in (t- \left(\mathcal{G}(\mathcal{H}(t))\right)^{-1}, t].
\eeq
Combining \eqref{Gre8} and \eqref{Gre9}, we obtain 

\beq \label{Gre10}
  t \mathcal{G}^3(\mathcal{H}(t))  \leq \kappa + \mathcal H(t)  \mathcal{G}^2(\mathcal{H}(t)) \leq 2\kappa.
\eeq

Hence

\beqs
\mathcal H(t)  \leq   \mathcal F^{-1} \left(\frac{2\kappa}{t}\right)\mathcal H(0),
\eeqs 
which finishes the proof.
\end{itemize}
\end{proof}

\begin{remark}
The results of Lemma \ref{Mo2} are not optimal and can be improved when $t\mathcal G^3(H(t)) $ tends 
to infinity as $t$ tends to infinity. 
\end{remark}

We provide here the proof of (\ref{stabsemiunif}).
\begin{theorem}{\cite[pp. 40,41]{battybis}}\label{thmstab}
Let $(T(t))_{t\geq 0}$ be a bounded $C_{0}$-semigroup on a Banach space $X$ with generator $A$. Assume that
\[
\sigma(A)\cap i\mathbb{R} = \varnothing,
\]
where $\sigma(A)$ denotes the spectrum of the operator $A$. Then
\[
\|T(t)A^{-1}\|\longrightarrow 0.
\qquad t\to + \infty.
\]
\end{theorem}

\begin{proof}
 The proof of Theorem \ref{thmstab}, given in \cite{battybis}, can also be derived from Ingham's Tauberian theorem \cite{ingham}. However, for the convenience of the reader, we provide here a simpler and more direct proof, given by Batty in \cite{battybis} and based on an argument due to Korevaar \cite{korevaar}. 

This method relies on suitable estimates of contour integrals and avoids some of the technicalities usually involved in Tauberian arguments. In particular, it provides a transparent and efficient framework for establishing the desired asymptotic behavior of the semigroup.

Although the operator-valued integrals appearing below are understood in the strong operator topology, all the corresponding estimates are carried out in the operator norm topology.

\medskip

Assume therefore that
\[
\sigma(A)\cap i\mathbb{R}=\varnothing.
\]
Let $R>0$, and denote by $\gamma_{+}$ and $\gamma_{-}$ respectively the left and right semicircles of the circle
\[
|z|=R.
\]
Let $\gamma'$ be a contour contained in
\[
\{z\in \rho(A):\Re z<0\},
\]
joining $i\mathbb{R}$ to $-i\mathbb{R}$, where $\rho(A)$ denotes the resolvent set of $A$.

\medskip

For $t>0$, define
\[
g_t(z)=\int_{0}^{t} e^{-sz}T(s)\,ds.
\]
Then $g_t$ is an entire operator-valued function on $\mathbb{C}$ and satisfies
\[
e^{tz}g_t(z)
=(zI-A)^{-1}\bigl(e^{tz}I-T(t)\bigr),
\qquad z\in \rho(A).
\]
Applying Cauchy's theorem, one obtains
\begin{align}
T(t)A^{-1}
&=
-\frac{1}{2\pi i}
\int_{\gamma_{+}\cup\gamma'}
\left(1+\frac{z^{2}}{R^{2}}\right)
(zI-A)^{-1}T(t)\,\frac{dz}{z}
\nonumber\\
&=
-\frac{1}{2\pi i}
\int_{\gamma_{+}}
\left(1+\frac{z^{2}}{R^{2}}\right)
(zI-A)^{-1}T(t)\,\frac{dz}{z}
\label{eq:contour}\\
&\quad
-\frac{1}{2\pi i}
\int_{\gamma'}
\left(1+\frac{z^{2}}{R^{2}}\right)
(zI-A)^{-1}e^{tz}\,\frac{dz}{z}
\nonumber\\
&\quad
-\frac{1}{2\pi i}
\int_{\gamma_{-}}
\left(1+\frac{z^{2}}{R^{2}}\right)
g_t(z)e^{tz}\,\frac{dz}{z}.
\nonumber
\end{align}

As $t\to +\infty$, the second integral in \eqref{eq:contour} converges to zero by the Dominated Convergence Theorem.

Moreover, on the contours $\gamma_{\pm}$, one has
\[
\left|1+\frac{z^{2}}{R^{2}}\right|
=\frac{2|\Re z|}{R}.
\]

On $\gamma_{+}$, using the boundedness of the semigroup, we estimate
\[
\|(zI-A)^{-1}T(t)\|
=
\left\|
\int_{0}^{+\infty} e^{-sz}T(s+t)\,ds
\right\|
\leq
\frac{M}{|\Re z|},
\]
where
\[
M:=\sup_{t\geq 0}\|T(t)\|<\infty.
\]
Similarly, on $\gamma_{-}$,
\[
\|e^{tz}g_t(z)\|
=
\left\|
\int_{0}^{t} e^{(t-s)z}T(s)\,ds
\right\|
\leq
\frac{M}{|\Re z|}.
\]
Consequently, both the first and third integrals in \eqref{eq:contour} are bounded in norm by $M/R$. Hence,
\[
\limsup_{t\to\infty}\|T(t)A^{-1}\|
\leq
\frac{2M}{R}.
\]
Since $R>0$ is arbitrary, letting $R\to + \infty$ yields
\[
\|T(t)A^{-1}\|\longrightarrow 0,
\qquad t\to + \infty.
\]

\end{proof}

\begin{proof}[Proof of Proposition \ref{ww}]
We have two cases:\\

\noindent Case 1: Assume that 
\begin{eqnarray} \label{exponential}
\|T(t)\|\longrightarrow 0,
\qquad t\to + \infty.
\end{eqnarray}
Therefore  
\begin{eqnarray*} 
\| T(t)  w_0\|_{X_1\times X}^2  \leq  \mathcal K(t) \|w_0\|_{X_1\times X}^2,\qquad \forall w_0 \in X_2\times X_1\setminus \left\{0\right\},
\end{eqnarray*}
with  
\[
\mathcal K(t) = \sup_{s\in (t,+\infty)}\| T(s)\|^2.
\]
\noindent Case 2: Let's assume now that \eqref{exponential} does not hold.
We deduce from Theorem \ref{thmstab} that 
\begin{eqnarray*} \label{stabsemiunifbis}
\| T(t) \mathcal A_d^{-1}\| \to 0,  \textrm{  as  } t \to +\infty.
\end{eqnarray*}
Hence 
\begin{eqnarray*} 
\| T(t)  w_0\|_{X_1\times X}^2  \leq  \mathcal K(t) \mu(w_0) \|w_0\|_{X_1\times X}^2,\qquad \forall w_0 \in X_2\times X_1\setminus \left\{0\right\},
\end{eqnarray*}

with  
\[
\mathcal K(t) = \sup_{s\in (t,+\infty)}\| T(s) \mathcal A_d^{-1}\|^2,
\]
which finishes the proof of the Proposition. \\

\noindent Notice that Case 2 always applies, while Case 1 is a special case of it.

\end{proof}

\end{document}